\documentclass[12pt]{amsart}
\usepackage{amsmath,amsthm,amsfonts,amssymb,times,latexsym,mathabx,url}

\newtheorem{theorem}{Theorem}[section]

\newtheorem{lem}[theorem]{Lemma}

\numberwithin{equation}{section}

\renewcommand{\a}{\alpha}

\renewcommand{\d}{\delta}

\renewcommand{\l}{\lambda}
\renewcommand{\o}{\omega}
\renewcommand{\O}{\Omega}

\newcommand{\N}{\mathbb{N}}

\newcommand{\s}{\sigma}

\newcommand{\Z}{\mathbb{Z}}

\renewcommand{\leq}{\leqslant}
\renewcommand{\geq}{\geqslant}

\begin{document}

\title{Numbers of the form $k+f(k)$}
\author{Mikhail R. Gabdullin, Vitalii V. Iudelevich, Florian Luca}

\address{Steklov Mathematical Institute,
	Gubkina str., 8, Moscow, Russia, 119333}
\email{gabdullin.mikhail@yandex.ru}

\address{Moscow State University, Leninskie Gory str., 1, Moscow, Russia, 119991}
\email{vitaliiyudelevich@mail.ru} 

\address{School of Maths, Wits University, Private Bag 3, Wits 2050, South Africa and Centro de Ciencias Matem\'aticas UNAM, Morelia, Mexico}
\email{florian.luca@wits.ac.za}

\begin{abstract}
For a function $f\colon \N\to\N$, let 
$$
N^+_f(x)=\{n\leq x: n=k+f(k) \mbox{ for some } k\}.
$$
Let $\tau(n)=\sum_{d|n}1$ be the divisor function, $\omega(n)=\sum_{p|n}1$ be the prime divisor function, and $\varphi(n)=\#\{1\leq k\leq n: \gcd(k,n)=1 \}$ be Euler's totient function. We show that
\begin{align*}
&(1) \quad x \ll N^+_{\o}(x), \\ 
&(2) \quad x\ll N^+_{\tau}(x) \leq 0.94x, \\
&(3) \quad x \ll N^+_{\varphi}(x) \leq 0.93x.
\end{align*}
\end{abstract}

\date{\today}
\maketitle

\section{Introduction}

For a function $f\colon \N\to\N$, let 
$$
N^+_f(x):=\{n\leq x: n=k+f(k) \mbox{ for some } k\}.
$$
In this paper, we are interested in lower and upper bounds for $N^+_f(x)$ for $f\in\{\tau,\o,\varphi \}$; here, $\o(n)$ denotes the number of prime divisors of a positive integer $n$, $\tau(n)$ is the divisor function, and $\varphi(n)=\#\{k\leq n: \gcd(k,n)=1\}$ is the Euler totient function. This question for $f=\o$ dates back to the work \cite{ESP3} (see also \cite{ESP1} and \cite{ESP2}) of Erd\H{o}s, S\'ark\"ozy, and Pomerance who showed (see Corollary after Theorem 3.1) that
$$
N^+_{\o}(x)\gg x.
$$ 
However, obtaining upper bounds for $N^+_{\o}(x)$ seems to be a much more complex problem. The best current upper bound  
$$
x-N^+_{\o}(x) \gg \frac{x}{\log\log x}
$$
is due to Kucheriaviy \cite{Kuch}. We also note that the problem of showing that there are infinitely many numbers which cannot be represented in the form $k+\tau(k)$ appeared in 2018 at the entrance exam into the Yandex School of Data Analysis; see \cite[Problem 8]{SDA}.

Numerical calculations predict that, for large $x$,
$$
N^+_{\o}(x)\approx 0.73x,
$$
$$
N^+_{\tau}(x)\approx 0.67x,
$$
$$
N^+_{\varphi}(x)\approx 0.37x.
$$
Our results are the following. Firstly, by a more general and transparent argument, we recover the aforementioned bound from \cite{ESP3}. 

\begin{theorem}\label{th1}
We have
$$
N^+_{\o}(x) \gg x.
$$	
\end{theorem}
 
The method works for other functions as well. 
 
\begin{theorem}\label{th2} 
We have	
$$
x\ll N^+_{\tau}(x)\leq 0.94x.
$$
\end{theorem}

The upper bound here is due to the non-uniformity of $k+\tau(k)$ modulo $3$, namely, the fact that 
$$
\#\{k\leq x: k+\tau(k)\equiv0\pmod3\}=
\left(\frac{1}{3}+\frac{\zeta(3)}{12\zeta(2)}+o(1)\right)x.
$$
 
We also establish the following general estimate. 
 
\begin{theorem}\label{th3} 
Let $f\colon \N\to \Z$ be a function such that $0\leq f(k)\leq ck$ for some $c>0$. Then 
$$
x-N^+_f(x)=\#\{n\leq x: n\neq k+f(k) \} \geq \frac{1}{(2c+2)x}\sum_{k\leq x}f(k).
$$	
\end{theorem}

The bound in this theorem is tight (up to the constant $(2c+2)^{-1}$) in general, as one can see from the examples of $f(k)\equiv1$ or $f(k)=k$. However, it is not tight at all, say, for 
$$
f(k):=\begin{cases}
1, \mbox{if $k$ is odd},\\
0, \mbox{if $k$ is even.}
\end{cases}
$$

One of the applications of Theorem \ref{th2} is the case $f=\varphi$. Since 
$$
\sum_{k\leq x}\varphi(k)=\left(\frac{3}{\pi^2}+o(1)\right)x^2,
$$ 
we get
$$
N^+_{\varphi}(x)\leq \left(1-\frac{3}{4\pi^2}+o(1)\right)x \leq 0.93x
$$
for large $x$. However, some bounds for the limit distribution of $\varphi(k)/k$ should definitely give a better result. Let us denote 
$$
\Phi_x(\l)=\frac1x\#\{k\leq x: \varphi(k)/k \leq\l\}.
$$
It is known that $\Phi(\l):=\lim_{x\to\infty}\Phi_x(\l)$ exists for each $\l\in[0,1]$, and that $\Phi(\l)$ is an increasing singular function. 


\begin{theorem}\label{th4} 
We have
$$
x \ll N^+_{\varphi}(x) \leq \left(\frac12+\int_0^1\frac{\Phi(t)dt}{(1+t)^2}+o(1)\right)x.
$$	
\end{theorem}

Our numerical calculations (based on the values of $\Phi_x(\l)$ for $x=10^5$ at the points $a/10$, where $a=0,...,9$) predict that $\int_0^1\frac{\Phi(t)dt}{(1+t)^2}<0.17$, which would imply that $N^+_{\varphi}(x)< 0.67x$.

\medskip 

Now we briefly discuss the proofs of the lower bounds for $N^+_f(x)$ in the cases $f\in\{\o,\tau, \varphi\}$. To obtain such a result, it is natural to estimate the ``additive energy''
$$
\#\{k,l\leq x: k+f(k)=l+f(l)\}.
$$
Once we show that this quantity is $O(x)$, we immediately get $N^+_f(x)\gg x$ by a standard argument which uses the Cauchy-Schwartz inequality (see the end of Section \ref{sec2} for details). However, it is technically easier to work with the image of an appropriate dense subset $A\subset\{n\leq x\}$ instead, and check that
$$
\#\{k_1,k_2\in A: k_1+f(k_1)=k_2+f(k_2)\} \ll x,
$$
which will imply in a similar manner that the image of $A$ under the map $k+f(k)$ has size $\gg x$, and, hence, that $N^+_f(x)\gg x$. We note that our lower bounds of the type $\gg x$ are explicit; however, the constants there are very small (of order $10^{-5}$ or even worse), so we did not try to write them down and  optimize. 

One of the remaining open questions in our setting is to establish that there are $\gg x$ numbers not exceeding $x$ which are not representable in the form $k+\o(k)$. Note that in the work \cite{ESP1}, the authors conjectured in the final remark that the number of $n\leq x$ which have at least two representations is $\gg x$; clearly, it implies the bound we believe in.  

In Sections \ref{sec2} and \ref{sec3} we prove Theorems \ref{th1} and \ref{th2}, respectively. Section \ref{sec4} is devoted to Theorem \ref{th3}, and in the final Section \ref{sec5} we establish Theorem \ref{th4}.

\section{Proof of Theorem \ref{th1}} \label{sec2} 

Let $y=x^{1/3}$ and $K>0$ be a fixed large enough number. For brevity we will use the notation $\log_2x:=\log\log x$, $\log_3x:=\log\log\log x$, etc. Let us consider the set $A$ of $n\leq x$ of the form $n=lp$, where 
\begin{itemize}
	\item[(i)] $l\leq y$;
	
	\item[(ii)] $l$ is square-free;
	
	\item[(iii)] $\o(l) \in I=[\log_2 x - K\log_2^{1/2}x,\log_2x+K\log_2^{1/2}x]$; 
	
	\item[(iv)] $y<p\leq x/l$ and $p$ is a prime.
\end{itemize}

\smallskip

We have
$$
|A|= \sum_{\substack{l\leq y \\ l \mbox{ \tiny{obeys (ii) and (iii)} } }} \left(\pi(x/l) - \pi(y)\right) \gg \sum_{\substack{l\leq y \\ l \mbox{ \tiny{obeys (ii) and (iii)} } }} \pi(x/l) \gg \frac{x}{\log x}   \sum_{\substack{l\leq y \\ l \mbox{ \tiny{obeys (ii) and (iii)} } }} \frac1l \, .
$$
Since for any $M$ with $y^{1/2}<M\leq y/2$ there are $\gg M$ numbers $l\in [M,2M)$ obeying (ii) and (iii) (here we use that $K$ is large), we get
\begin{equation}\label{2.1}
|A|\gg x.
\end{equation}

Our goal is to bound the quantity 
$$
E:=\#\{(n,m)\in A^2: n+\o(n)=m+\o(m) \}.
$$
Note that
\begin{equation}\label{2.2}
E=|A|+R,	
\end{equation}
where $R$ corresponds to the solutions with $n\neq m$. We have
$$
R\leq \sum_{\substack{l_1, l_2\leq y \\ l_1\neq l_2}}\#\{(p_1,p_2): l_1p_1+\o(l_1) = l_2p_2 +\o(l_2) \} = \sum_{\substack{l_1, l_2\leq y \\ l_1\neq l_2}} I_{l_1,l_2}, 
$$
where $I_{l_1,l_2}$ is the number of solutions (in primes $p_1\leq x/l_1$ and $p_2\leq x/l_2$) of the equation
\begin{equation}\label{2.3} 
l_1p_1-l_2p_2=N,	
\end{equation}
where $N=N(l_1,l_2)=\o(l_2)-\o(l_1)\neq0$. Fix any $l_1<l_2\leq y$ obeying (ii) and (iii), and let $\d:=\gcd(l_1,l_2)$. Now note that the equation (\ref{2.3}) has no solutions if $N$ is not divisible by $\d$, and otherwise is equivalent to   
\begin{equation}\label{2.4} 
	(l_1/\d)p_1-(l_2/\d)p_2=N/\d. 
\end{equation}
In the latter case, we have $\gcd(N\d^{-1},\,l_1l_2\d^{-2})=1$ (otherwise it again has no solutions). Also $|N|\leq z:=2K\log_2^{1/2}x$ by the definition of the set $A$, and thus we may consider only those $l_1,l_2$ for which $\d\leq z$. Let $\s(n)$ stand for the sum of divisors of a positive integer $n$. For $a,b\leq x^{1/3}$ with $\gcd(a,b)=\gcd(ab,N)=1$, Selberg's sieve implies that the equation $ap_1-bp_2=N$ has at most
$$
\ll \frac{x}{ab(\log x)^2} \prod_{p|Nab, p>2}\frac{1-1/p}{1-2/p} \ll  \frac{x}{\varphi(a)\varphi(b)(\log x)^2}\frac{\s(|N|)}{|N|}
$$
solutions in primes $p_1\leq x/a$, $p_2\leq x/b$, with the implied constant being absolute. Now we apply this for the situation of (\ref{2.4}) and first note that the condition $p_s\leq x/l_s$ is equivalent to $p_s\leq (x/\d)/(l_s/\d)$ for $s=1,2$, so we can take $x/\d$ instead of $x$ in that bound. 
We get
$$
I_{l_1,l_2} \ll \frac{x/\d}{(\log x)^2}\frac{\s(|N|/\d)}{\varphi(l_1/\d)\varphi(l_2/\d)(|N|/\d)}.
$$
Since $\s(a)/a \leq \s(b)/b$ whenever $a\mid b$, we get
$$
R\ll \frac{x}{(\log x)^2}\sum_{\d\leq z}\frac{1}{\delta}\sum_{\substack{l_1,l_2\leq y, \, \gcd(l_1,l_2)=\d \\ \d|\o(l_2)-\o(l_1)\neq0}} \frac{\s(|N(l_1,l_2))|)}{\varphi(l_1/\d)\varphi(l_2/\d)|N(l_1,l_2)|},
$$
where the inner summation is over $l_1$ and $l_2$ satisfying also (ii) and (iii). Let $l_s':=l_s/\d$ for $s=1,2$. Since $l_1$ and $l_2$ are square-free, we have 
$$
\o(l_1')-\o(l_2')=\o(l_1/\d)-\o(l_2/\d)=\o(l_1)-\o(l_2).
$$ 
Let $i:=\o(l_1')$ and $j:=\o(l_2')$. Since $\d \ll (\log_2 x)^{1/2}$, we trivially have $\o(\d)\ll \log_3 x$. Then, for large enough $x$,
$$
i,j\in I':=[\log_2 x-(K+1)(\log_2 x)^{1/2}, \, \log_2 x+(K+1)(\log_2 x)^{1/2}]
$$
and
$$
R\ll \frac{x}{(\log x)^2}\sum_{\d\leq z}\frac{1}{\delta}\sum_{\substack{i,j\in I' \\ i\neq j,\, \d|i-j}} \frac{\s(|i-j|)}{|i-j|}\sum_{\substack{l_1', l_2'\leq y/\d \\ \gcd(l_1',l_2')=1 \\ \o(l_1')=i, \, \o(l_2')=j}} \frac{1}{\varphi(l_1')\varphi(l_2')}.
$$
Further, note that 
\begin{equation*}
\sum_{\substack{l\leq y/\d \\ \o(l)=j}}\frac{1}{\varphi(l)}\leq \frac{1}{j!}\left(\sum_{p\leq y/\d}\frac{1}{\varphi(p)} \right)^j \leq \frac{u^j}{j!},
\end{equation*}
where $u=\log_2x+B$ with some absolute constant $B>0$. Forgetting about the condition $\gcd(l_1',l_2')=1$, we thus get 
\begin{equation}\label{2.5} 
R \ll \frac{x}{(\log x)^2}\sum_{\d\leq z}\frac{1}{\delta}\sum_{\substack{i,j\in I' \\ i>j,\, \d|i-j}} \frac{\s(i-j)}{i-j}\frac{u^{i+j}}{i!j!} = \frac{x}{(\log x)^2}R_1,
\end{equation}
where
\begin{equation}\label{2.6}
R_1 = \sum_{\substack{i,j\in I' \\ i > j}}\frac{\s(i-j)}{i-j} \frac{u^{i+j}}{i!j!} \sum_{\d| i-j} \frac{1}{\d}=\sum_{\substack{i,j\in I' \\ i>j}} \frac{u^{i+j}}{i!j!} \left(\frac{\s(i-j)}{i-j}\right)^2.
\end{equation}
Note that for any $i\in I'$
$$
\frac{u^i}{i!}\left(\frac{u^{i+1}}{(i+1)!}\right)^{-1}=\frac{i+1}{u}= 1+O((\log_2 x)^{-1/2}).
$$
Thus, we obtain
$$
\frac{u^i}{i!}\left(\frac{u^{j}}{j!}\right)^{-1} = \left(1+O((\log_2x)^{-1/2})\right)^{i-j} \asymp 1
$$
uniformly on $i,j\in I'$. Since
$$
\sum_{i\in I'}\frac{u^i}{i!} \leq \sum_{i=0}^{\infty}\frac{u^i}{i!}=e^u\ll \log x,
$$
and $|I'|\asymp (\log_2x)^{1/2}$, we see that
$$
\frac{u^i}{i!} \ll \frac{\log x}{(\log_2x)^{1/2}}
$$
uniformly on $i\in I'$. Therefore (\ref{2.5}) and (\ref{2.6}) imply
$$
R\ll \frac{x}{\log_2x}\sum_{i,j\in I': i>j}\left(\frac{\s(i-j)}{i-j}\right)^2 \ll \frac{x}{(\log_2x)^{1/2}}\sum_{r\leq 2(K+1)(\log_2x)^{1/2}}\left(\frac{\s(r)}{r}\right)^2 \ll x
$$
by the well-known fact that $\s(r)/r$ is $O(1)$ on the average. Thus, by (\ref{2.2})
$$
E \ll x
$$ 
as well. Now denote $M:=\max\{\o(k): k\leq x\}$. Now let 
$$
B=\{n\leq x+M: n=k+\o(k) \mbox { for some } k\in A \},
$$ 
and $r(n)=\#\{k\in A: k+\o(k)=n \}$. Then by (\ref{2.1}) and the Cauchy-Schwarz inequality 
$$
x \ll |A|=\sum_{n\in B}r(n) \leq |B|^{1/2}E^{1/2}, 
$$
and so $|B| \gg x$. Finally, since $M=x^{o(1)}$, we see that the number of $n\leq x$ representable in the form $k+\o(k)$ is at least $\gg x$. This completes the proof.

\section{Proof of Theorem \ref{th2}}  \label{sec3}

\subsection{Proof of the lower bound} We argue similarly to the proof of Theorem \ref{th1} and omit the details which do not require any changes. Let $y=x^{1/3}$ and $K>0$ be large enough. Let us consider the set $A$ of $n\leq x$ such that $n=lp$, where 
\begin{itemize}
\item[(i)] $l\leq y$;
	
\item[(ii)] $l$ is odd and square-free;
	
\item[(iii)] $\o(l) \in I=[\log_2 x - K\log_2^{1/2}x,\log_2x+K\log_2^{1/2}x]$; 
	
\item[(iv)]  $y<p\leq x/l$ and $p$ is a prime.
\end{itemize}

\smallskip

We again have
\begin{equation}\label{3.1}
|A|\gg x.
\end{equation}
Let
$$
E:=\#\{(n,m)\in A^2: n+\tau(n)=m+\tau(m) \}.
$$
Then
\begin{equation}\label{3.2}
E=|A|+R,	
\end{equation}
where $R$ corresponds to the solutions with $n\neq m$. We have
$$
R\leq \sum_{\substack{l_1, l_2\leq y \\ l_1\neq l_2}}\#\{(p_1,p_2): l_1p_1+2^{\o(l_1)+1} = l_2p_2 +2^{\o(l_2)+1} \} = \sum_{\substack{l_1, l_2\leq y \\ l_1\neq l_2}} I_{l_1,l_2}, 
$$
where $I_{l_1,l_2}$ is the number of solutions (in  primes $p_1\leq x/l_1$ and $p_2\leq x/l_2$) of the equation
\begin{equation}\label{3.3} 
l_1p_1-l_2p_2=N,	
\end{equation}
where $N=N(l_1,l_2)=2^{\o(l_2)+1}-2^{\o(l_1)+1}\neq0$. Fix any $l_1<l_2\leq y$ obeying (ii) and (iii), and let $\d:=\gcd(l_1,l_2)$. Now note that the equation (\ref{3.3}) has no solutions if $N$ is not divisible by $\d$, and otherwise is equivalent to   
\begin{equation}\label{3.4} 
(l_1/\d)p_1-(l_2/\d)p_2=N/\d. 
\end{equation}
In the latter case we have $\gcd(N\d^{-1},\,l_1l_2\d^{-2})=1$ (otherwise it again has no solutions). Also $|N|\leq z:=\log x$ (say) by the definition of the set $A$, and thus we may consider only those $l_1,l_2$ for which $\d\leq z$. Now Selberg's sieve implies that 
$$
I_{l_1,l_2} \ll \frac{x/\d}{(\log x)^2}\frac{\s(|N|)}{\varphi(l_1/\d)\varphi(l_2/\d)|N|},
$$
and therefore
$$
R\ll \frac{x}{(\log x)^2}\sum_{\d\leq z}\frac{1}{\delta}\sum_{\substack{l_1,l_2\leq y, \, \gcd(l_1,l_2)=\d \\ \d|2^{\o(l_2)+1}-2^{\o(l_1)+1}\neq0}} \frac{\s(|N(l_1,l_2))|)}{\varphi(l_1/\d)\varphi(l_2/\d)|N(l_1,l_2)|};
$$
the inner summation here is over $l_1,l_2$ satisfying also (ii) and (iii). Note that $\d=\gcd(l_1,l_2)$ is odd, and then $\d\, |\, 2^{\o(l_2)+1}-2^{\o(l_1)+1}$ is equivalent to $\d\, |\, 2^{|\o(l_2)-\o(l_1)|}-1$. Let $l_s':=l_s/\d$ for $s=1,2$. 
Since $l_1$ and $l_2$ are square-free, we have 
$$
\o(l_1')-\o(l_2')=\o(l_1/\d)-\o(l_2/\d)=\o(l_1)-\o(l_2).
$$ 
Let $i:=\o(l_1')$ and $j:=\o(l_2')$. Since $\d$ divides $2^{|\o(l_2)-\o(l_1)|}-1$, which is $\exp(O(\log_2x)^{1/2})$ by the condition (iii), and $\o(n) \ll \log n/\log_2 n$, we have
$$
\o(\d)\ll \frac{(\log_2 x)^{1/2}}{\log_3x}, 
$$
and so, for $x$ large, 
$$
i,j\in I':=[\log_2 x-(K+1)(\log_2 x)^{1/2}, \, \log_2 x+(K+1)(\log_2 x)^{1/2}].
$$
Since $\s(ab)\leq\s(a)\s(b)$,
$$
R\ll \frac{x}{(\log x)^2}\sum_{\d\leq z}\frac{1}{\delta}\sum_{\substack{i,j\in I' \\ i> j,\, \d|2^{i-j}-1}} \frac{\s(2^{i-j}-1)}{2^{i-j}-1}\sum_{\substack{l_1', l_2'\leq y/\d \\ \gcd(l_1',l_2')=1 \\ \o(l_1')=i, \, \o(l_2')=j}} \frac{1}{\varphi(l_1')\varphi(l_2')}.
$$
Setting $u=\log_2x+B$ (here $B>0$ is an absolute constant), we get 
\begin{equation}\label{3.5} 
R \ll \frac{x}{(\log x)^2}\sum_{\d\leq z}\frac{1}{\delta}\sum_{\substack{i,j\in I' \\ i>j,\, \d|2^{i-j}-1}} \frac{\s(2^{i-j}-1)}{2^{i-j}-1}\frac{u^{i+j}}{i!j!}.
\end{equation}
As before,
$$
\frac{u^i}{i!} \ll \frac{\log x}{(\log_2x)^{1/2}}
$$
uniformly on $i\in I'$. Therefore, (\ref{3.5}) implies
$$
R\ll \frac{x}{\log_2x}\sum_{i,j\in I': i>j}\left(\frac{\s(2^{i-j}-1)}{2^{i-j}-1}\right)^2 \ll \frac{x}{(\log_2x)^{1/2}}\sum_{r\leq 2(K+1)(\log_2x)^{1/2}}\left(\frac{\s(2^r-1)}{2^r-1}\right)^2.
$$
Now we need the fact that
$$
\sum_{r\leq y}\left(\frac{\s(2^r-1)}{2^r-1}\right)^2 \ll y,
$$
which was proved in \cite{LS}. Actually the main result of \cite{LS} is a much more general statement: it states that the average value of the moments of $(\sigma(n)/n)^k$ (with a positive integer $k$), when $n$ ranges over members of a linearly recurrent sequence of integers like $2^r-1$, is bounded. 

Thus, $R\ll x$, and, by (\ref{3.2})
$$
E \ll x
$$ 
as well. Since $\max\{\tau(k): k\leq x\}=x^{o(1)}$, we see from (\ref{3.1}) and the Cauchy-Schwarz inequality that the number of $n\leq x$ representable in the form $k+\tau(k)$ is at least $\gg x$. The claim follows.

\subsection{Proof of the upper bound} We claim that it is enough to show that
\begin{equation}\label{3.6}
T(x):=  \#\{k\leq x: k+\tau(k)\equiv\, 0\,(\text{mod}\, 3)\} = (c+o(1))x, 
\end{equation}
where $c = \frac13+\frac{\zeta(3)}{12\zeta(2)} = 0.394\ldots$ Indeed, since $\max_{k\leq x}\tau(k)=x^{o(1)}$, the same asymptotics holds if we additionally require $k+\tau(k)\leq x$. Then the images (under the map $k\mapsto k+\tau(k)$) of the remaining $(1-c+o(1))x$ numbers $k\leq x$ cannot cover all $2x/3+O(1)$ numbers $n\leq x$ which are equal $\pm1\pmod3$. Hence, at least $(\frac{\zeta(3)}{12\zeta(2)}-o(1))x\geq 0.06x$ of these $n$ have no representation of the form $k+\tau(k)$, as desired. 

So it remains to prove (\ref{3.6}). We write $T(x)$ as
$$
T(x):= T_0(x)+T_1(x)+T_2(x),
$$
where 
$$
T_i(x): = \#\{k\leq x: k\equiv i \, (\text{mod}\, 3) \mbox{\, and \,} \tau(k)\equiv -i \,(\text{mod}\, 3) \}, \qquad i\in\{0,1,2\}.
$$
First, we derive the asymptotic formula for $T_0(x).$ We have
\begin{equation*}
T_0(x) = \sum_{\substack{n\leq x \\ n\equiv\,0\,(\text{mod}\, 3)}}1 - \sum_{\substack{n\leq x \\ n\equiv\,0\,(\text{mod}\, 3) \\ \tau(n)\not{\equiv}\,0\,(\text{mod}\, 3)}}1\\
=\frac{x}{3}-\sum_{\substack{\alpha\geq 1\\ \alpha\not{\equiv}\,2\,(\text{mod}\, 3)}}K\left(\frac{x}{3^\alpha}\right)  + O(1),
\end{equation*}
where 
$$
K(y) := \#\{k\leq y: \gcd(k,3) = 1, \, \tau(k)\not{\equiv}\,0\,(\text{mod}\, 3)\}.
$$
To calculate $K(y)$, we consider the corresponding Dirichlet series
$$
\mathcal{K}(s) := \sum_{\substack{k = 1 \\ \gcd(k,3) = 1 \\ \tau(k)\not{\equiv}\,0\,(\text{mod}\, 3)}}^{+\infty}\frac{1}{k^s}, \quad s\in {\mathbb C},~{\text{\rm Re}}(s)>1.
$$
We have
\begin{multline*}
\mathcal{K}(s) = \prod_{p\,\nmid\, 3}\left( 1+ \frac{1}{p^s}+\frac{1}{p^{3s}}+\frac{1}{p^{4s}}+\frac{1}{p^{6s}}+\cdots\right) \\
= \zeta(s)\prod_p\left\{\left( 1-\frac{1}{p^s}\right)\left( 1+ \left(\frac{1}{p^s}+\frac{1}{p^{3s}} \right)\frac{1}{1-p^{-3s}} \right)   \right\}\\ 
\times\left(1+\left(\frac{1}{3^s}+\frac{1}{3^{3s}}\right)\frac{1}{1-3^{-3s}} \right)^{-1}
=\frac{\zeta(s)\zeta(3s)}{\zeta(2s)}\frac{1-27^{-s}}{1+3^{-s}}.
\end{multline*}
Applying Lemma 3.1 from \cite{Changa} to $\mathcal{K}(s)$, we obtain
$$
K(y) = \frac{13\zeta(3)}{18\zeta(2)}x+O\left(\frac{x}{\log x}\right).
$$
Hence,
\begin{equation}\label{3.7}
T_0(x)=\frac{x}{3}-\frac{13\zeta(3)}{18\zeta(2)}\,x \sum_{\substack{\alpha = 1 \\ \alpha\not{\equiv}\,2\,(\text{mod}\, 3)}}^{\infty}\frac{1}{3^{\alpha}} + o(x) = \left(\frac13-\frac{5\zeta(3)}{18\zeta(2)}+o(1)\right)x.
\end{equation}
Now we obtain the asymptotics for $T_1(x)+T_2(x)$. Note that for any with $\tau(k) \not{\equiv}\, 0\,(\text{mod}\,3)$, 
\begin{equation*}
\tau(k) = \prod_{p^\a || k}(\a+1) \equiv \prod_{\substack{p^{\a}||k \\ \a \equiv 1 (\text{mod}\, 3)}}2 \equiv 2^{W(k)} (\text{mod}\, 3),
\end{equation*}
where 
$$
W(k) := \#\{p|k: p^{\a}||k \mbox{ \, for some } \a\equiv 1 \, (\text{mod}\, 3)\}.
$$
Hence, for such $k$, the congruence $\tau(k) \equiv\, 2\,(\text{mod}\,3)$ holds if and only if $W(k) \equiv\, 1\,(\text{mod}\,2)$. Let $\chi_0$ be the principal character $\text{mod}\, 3$ and $\chi_3$  be the only nontrivial character $\text{mod}\, 3$. Then 
\begin{multline*}
T_1(x) = \#\{k \leq x: k \equiv 1 \, (\text{mod}\, 3), \, \tau(k) \not{\equiv} 0\, (\text{mod}\, 3), \, W(k) \equiv\, 1\,(\text{mod}\,2)\} \\
=\frac14\sum_{\substack{k\leq x \\ \tau(k) \not{\equiv} 0 (\text{mod}\, 3)}}\left( \chi_0(k)+\chi_3(k)\right)\left( 1-(-1)^{W(k)}\right)\\
= \frac14\left(T_1^{(1)}(x)-T_1^{(2)}(x)+T_1^{(3)}(x)-T_1^{(4)}(x)\right), 
\end{multline*}
where
$$
T_1^{(1)}(x) := \sum_{\substack{k \leq x \\ \gcd(k,3) = 1 \\ \tau(k) \not{\equiv} 0 (\text{mod}\, 3)}} 1,\ \ \ \ \ \ \ \ \ \ \ T_1^{(2)}(x) := \sum_{\substack{k \leq x \\ \gcd(k,3) = 1 \\ \tau(k) \not{\equiv} 0 (\text{mod}\, 3)}} (-1)^{W(k)},
$$
and
$$
T_1^{(3)}(x) := \sum_{\substack{k \leq x \\ \gcd(k,3) = 1 \\ \tau(k) \not{\equiv} 0 (\text{mod}\, 3)}} \chi_3(k),\ \ \ \ \ \ \ \ \ \  T_1^{(4)}(x) := \sum_{\substack{k \leq x \\ \gcd(k,3) = 1 \\ \tau(k) \not{\equiv} 0 (\text{mod}\, 3)}} \chi_3(k)(-1)^{W(k)}.$$
Arguing similarly, we obtain
$$
T_2(x) = \frac14\left(T_2^{(1)}(x)+T_2^{(2)}(x)-T_2^{(3)}(x)-T_2^{(4)}(x)\right).
$$
Hence,
\begin{equation}\label{3.8}
T_1(x)+T_2(x) = \frac12\left(T_1^{(1)}(x)-T_1^{(4)}(x) \right).
\end{equation}
Since the sum $T_1^{(1)}(x)$ coincides with $K(x)$, it follows that
\begin{equation}\label{3.9}
T_1^{(1)}(x) = \left(\frac{13\zeta(3)}{18\zeta(2)}+o(1)\right)x.
\end{equation}
It remains to estimate $T_1^{(4)}(x)$. As before, we consider the corresponding Dirichlet series 
$$
\mathcal{T}(s) := \sum_{\substack{k=1 \\ \tau(k) \not{\equiv} 0 (\text{mod}\, 3)\\ }}^{\infty}\frac{\chi_3(k)(-1)^{W(k)}}{k^s}, \quad s\in {\mathbb C},~~{\text{\rm Re}}(s)>1.
$$
We have
\begin{multline*}
\mathcal{T}(s) = \prod_{p\equiv 1(\text{mod}\,3)}\left(1+\sum_{\substack{\nu = 1 \\ \nu\not{\equiv}2(\text{mod}\,3)}}^{\infty}\frac{(-1)^{W(p^\nu)}}{p^{\nu s}} \right) \prod_{p\equiv 2(\text{mod}\,3)}\left(1+\sum_{\substack{\nu = 1 \\ \nu \not{\equiv}2(\text{mod}\,3)}}^{\infty}\frac{(-1)^{W(p^\nu)}(-1)^\nu}{p^{\nu s}} \right)\\
= \prod_{p\equiv 1(\text{mod}\,3)}\left(1+\left( -\frac{1}{p^s}+\frac{1}{p^{3s}}\right)\frac{1}{1-p^{-3s}} \right) \prod_{p\equiv 2(\text{mod}\,3)}\left(1+\left( \frac{1}{p^s}-\frac{1}{p^{3s}}\right)\frac{1}{1+p^{-3s}} \right)\\
=L^{-1}(s,\chi_3) \prod_{p\equiv 1(\text{mod}\,3)}\left(1+\frac{1}{p^{3s}-1} \right)\prod_{p\equiv 2(\text{mod}\,3)}\left(1-\frac{1}{p^{3s}+1} \right)=L^{-1}(s,\chi_3)L(3s,\chi_3).  
\end{multline*}
By \cite{Kuch}, Lemma 10, we obtain
\begin{equation}\label{3.10}
T_1^{(4)}(x) \ll x \exp(-c\,\sqrt{\log x})
\end{equation}
for some absolute constant $c>0$. Combining (\ref{3.7})--(\ref{3.10}), we complete the proof of (\ref{3.6}).

\section{Proof of Theorem \ref{th3}}  \label{sec4}

We use some ideas from \cite{Z}. For an integer $s\geq0$, we set 
$$
A_s(x):=\{n\leq x: \mbox{the equation } n=k+f(k) \mbox{ has exactly $s$ solutions} \}.
$$
Note that for any positive integer $x$ we have (since $n\leq x$ and $n=k+f(k)$ imply $k\leq x$)
$$
\bigsqcup_{s\geq0}A_s(x)=\{1,...,x\}=:[x] \supseteq \bigsqcup_{s\geq1}\bigsqcup_{n\in A_s(x)}\{k: n=k+f(k)\},
$$	
and, hence,
$$
\sum_{s\geq0}|A_s(x)| \geq \sum_{s\geq 1}s|A_s(x)|,
$$
or
\begin{equation}\label{4.1} 
 \sum_{s\geq2}(s-1)|A_s(x)|\leq |A_0(x)|.
\end{equation}	

 Now for each $n\in A_1(x)$ we fix the unique $k$ such that $k+f(k)=n$. Let us denote the set of these $k$ by $B_1(x)$. We have 
$$
\sum_{n\in A_1(x)}n = \sum_{k\in B_1(x)}(k+f(k)) = \sum_{k\leq x}(k+f(k)) - \sum_{\substack{ k \leq x \\ k\notin B_1(x)} } (k+f(k)).  
$$
Hence, 
$$
\sum_{k\leq x}f(k) = \sum_{\substack{ k \leq x \\ k\notin B_1(x)} }(k+f(k)) - \sum_{\substack{ n\leq x\\ n\notin A_1(x)}}n \leq (c+1)x(x - |B_1(x)|).
$$
But (\ref{4.1}) implies $x-|B_1(x)|=x-|A_1(x)|=|A_0(x)|+\sum_{s\geq2}|A_s(x)|\leq 2|A_0(x)|$. The claim follows.

\section{Proof of Theorem \ref{th4}} \label{sec5} 

\subsection{Proof of the lower bound}  

We will follow the method from \cite{LP}, where it was established that the set of numbers representable in the form $\sigma(k)-k$ has positive lower density, and bring up some modifications and simplifications.

We first provide some auxiliary statements. Here and in what follows $P^+(m)$ stands for the largest prime divisor of $m$. Letters $p,q,r,\pi$ will be used for primes. For a positive integer $m$ and a number $y\geq 2$, we also define the $y$-smooth part of $m$ by
$$
D(m,y):=\max\{t|m: P^+(t)\leq y\}.
$$
Let $x$ be large enough and, for the rest of the proof,
$$
y:=\frac{\log_2x}{\log_3x} \, .
$$

\begin{lem}\label{lem5.1} 
For all but $o(x)$ numbers $m\leq x$ we have
$$	
D(m,y)\leq \log x .
$$	
\end{lem}

\begin{proof}
This immediately follows from Theorem 07 of \cite{HT}.
\end{proof}

The next assertion is very similar to Lemma 2.1 of \cite{LP} and is related to the fact that the values of Euler's function is very ``composite''. 

\begin{lem}\label{lem5.2} 
All but $o(x)$ numbers $m\leq x$ obey
\begin{enumerate}
\item $p^{2a}|\varphi(m)$ for every prime power $p^a\leq y$;
\item $P^+(\gcd(m,\varphi(m)))\leq y$;
\item for each prime $p\leq y$, $m$ is not divisible by the prime powers $p^a$ greater than $y$; 
\item $D(m+\varphi(m),y)=D(m,y)$.
\end{enumerate}
\end{lem}

\begin{proof} We start with the first claim. Let $t\leq y$ be a fixed power of a prime and let $\mathcal{P}_t(y_1, \sqrt x)$, where $y_1:=\log x$, denote the set of primes $p \equiv 1 \pmod d$ with $p \in (y_1, \sqrt x]$. Clearly, the set of the integers $m\leq x$ with $d^2 \nmid \varphi(m)$ is contained in the set of those $m \leq x$ which are not divisible by two different primes $p_1,p_2 \in \mathcal{P}_t(y_1,\sqrt x)$. Note that, since $t\leq \log y_1$, the prime number theorem for arithmetic progressions implies that
$$
\sum_{p\in \mathcal{P}_t(y_1,\sqrt x)}\frac1p = \frac{\log(\log x/\log y_1)}{\varphi(t)}+O(1)
$$	
uniformly for $t \leq y$. The number of $m \leq x$ which are not divisible by two
different primes in $\mathcal{P}_t(y_1, \sqrt x)$ is, by the sieve (see \cite[Theorem 2.2]{HR}),
\begin{align*} 
\ll x\left(1 +\sum_{p\in \mathcal{P}_t(y_1,\sqrt x)}\frac1p\right)
\prod_{p\in \mathcal{P}_t(y_1,\sqrt x)}\left(1-\frac1p\right) \ll
\frac{x \log \log x}{\varphi(t)}\exp\left(-\frac{\log(\log x/ \log y_1)}{\varphi(t)}\right) \leq \\
\frac{x \log \log x}{\varphi(t)}\exp\left(-\frac{\log(\log x/ \log y_1)}{t}\right) \ll 
\begin{cases}
\frac{x}{\varphi(t)}, \mbox { if } \frac12y < t \leq y, \\
\frac{x}{\varphi(t) \log \log x}, \mbox{ if } t \leq \frac12y.
\end{cases}	
\end{align*}
Letting $t$ run over primes and powers of primes, we see that the number of
integers $m \leq x$ which do not have the property in (1) is $\ll x/ \log y = o(x)$.
	
The second claim is actually \cite[Theorem 8]{ELP}, which asserts that $\gcd(m, \varphi(m))$ is equal to $D(m,\log \log x)$ for all but $o(x)$ numbers $m\leq x$. Now it suffices to note that the number of $m\leq x$ divisible by a prime in $(y, \log \log x]$ is $o(x)$.
	
Now we prove the third claim. For each prime $p\leq y$, let $a=a(p)$ be the exponent such that $p^a\leq y<p^{a+1}$. Then the number of $m\leq x$ which is divided by $p^{a+1}$ for some $p\leq y$ is at most 
$$
\sum_{p\leq y}\frac{x}{y}\ll \frac{x}{\log y} = o(x).
$$

Finally, we check the fourth claim. For a positive integer $n$, let $\nu_p(n)$ be the largest exponent $k$ such that $p^k$ divides $n$. From the first and third claims we see that for all but $o(x)$ numbers $m\leq x$ the inequality 
$$
\nu_p(\varphi(m))>\nu_p(m)
$$ 
holds for each prime $p\leq y$, and thus $\nu_p(m)=\nu_p(m+\varphi(m))$ for these $m$. The claim follows.
\end{proof}

\begin{lem}\label{lem5.3} 
All but $o(x)$ numbers $m\leq x$ obey
$$
\sum_{\substack{p|\varphi(m)\\ p> \log_2 x}}\frac1p \leq 1.
$$	
\end{lem}

\begin{proof}
This follows from \cite[Lemma 5]{DL}; for completeness, we provide a short proof here. For a prime $p\leq x$, let $\pi(x;p,1)$ be the number of primes not exceeding $x$ which are congruent to $1\mod p$. Then, by partial summation and the Brun-Titchmarsh inequality,
\begin{multline*} 
\sum_{\substack{m\leq x \\ p|\varphi(m)}}1 \leq \frac{x}{p^2}+\sum_{\substack{q\leq x \\ p| q-1}} \frac xq = \frac{x}{p^2}+\frac{\pi(x;p,1)}{x}+\int_p^x\frac{\pi(u;p,1)du}{u^2} \ll\\
\frac{x}{p^2}+\frac1p + \frac1p\int_p^x\frac{du}{u\log((2u)/p)} \ll \frac{x}{p^2}+\frac1p\int_2^{2x/p}\frac{dt}{t\log t}
\ll \frac{x}{p^2} + \frac{x\log_2 x}{p} \ll \frac{x\log_2 x}{p}.
\end{multline*}
Thus,
$$
\sum_{m\leq x}\sum_{\substack{p|\varphi(m)\\ p> \log_2 x}}\frac1p = \sum_{p> \log_2 x}\frac1p\sum_{\substack{m\leq x \\ p|\varphi(m)}}1 \ll \sum_{p> \log_2 x}\frac{x\log_2 x}{p^2} \ll \frac{x}{\log_3 x},
	$$
	and the claim now follows from Markov's inequality.
\end{proof}

We will also need the fact that, roughly speaking, for a positive proportion of $m\leq x$, the number $m+\varphi(m)$ does not have too many prime factors. 

\begin{lem}\label{lem5.4} 
Let $C>1$ be fixed and $x$ be large enough depending on $C$. Then all but $O(x/C)$ numbers $m\leq x$ with $P^+(m)>m^{1/2}$ obey
$$
\sum_{\pi | m+\varphi(m)}\frac{1}{\pi} \leq C .  
$$	
\end{lem}

\begin{proof}
Clearly, we may discard the numbers $m\leq x^{1/2}$. Let $\O$ denote the set of $x^{1/2}<m\leq x$ with $P^+(m)>x^{1/4}$ such that
$$
\pi | \varphi(m) \mbox{ for all } \pi\leq y
$$
and
$$
P^+(\gcd(m,\varphi(m)))\leq y.
$$	
By Lemma \ref{lem5.2}, all but $o(x)$ numbers $m\leq x$ enjoy the last two displayed properties. Thus, it is enough to show that	
	$$
	\#\Big\{m\in \O: \sum_{\pi | m+\varphi(m)}\frac{1}{\pi} >C \Big\} \ll \frac{x}{C}.
	$$	
	By Markov's inequality, it will follow from
	\begin{equation}\label{1}
		S_1+S_2+S_3 \ll x,
	\end{equation}
	where
	$$
	S_1=\sum_{m\in\O}\sum_{\substack{\pi| m+\varphi(m)\\ \pi\leq y}} \frac{1}{\pi},
	$$
	$$
	S_2=\sum_{m\in\O}\sum_{\substack{\pi| m+\varphi(m)\\ y<\pi\leq x^{1/5}}} \frac{1}{\pi},
	$$
	and
	$$
	S_3=\sum_{m\in\O}\sum_{\substack{\pi| m+\varphi(m)\\ \pi>x^{1/5}}} \frac{1}{\pi}.
	$$
	Firstly, we estimate $S_1$. Note that for $m\in \O$ and $\pi\leq y$, the condition $\pi| m+\varphi(m)$ is equivalent to $\pi|m$. Thus,
	$$
	S_1 =\sum_{\pi\leq y}\frac{1}{\pi}\#\{m\in \O: \pi|m\} \ll x\sum_{\pi}\frac{1}{\pi^2} \ll x. 
	$$
	We also trivially have
	$$
	S_3 \leq \sum_{x^{1/5}< \pi \leq 2x}\frac{\#\O}{\pi} \ll x.
	$$
	It remains to show that $S_2\ll x$, and (\ref{1}) will follow. We write $m\in \O$ as $m=lq$, where $l<x^{3/4}$ and $q=P^+(m)\in (x^{1/4},x/l]$. We note that if a prime $\pi>y$ divides $m+\varphi(m)=(l+\varphi(l))q-\varphi(l)$, then $\gcd(l+\varphi(l),\pi)=1$: indeed, otherwise $\pi$ divides both $\varphi(l)+l$ and $\varphi(l)$, and then $\pi|\gcd(m,\varphi(m))$, which contradicts the fact that $m\in \O$. Thus, the condition $\pi| (l+\varphi(l))q-\varphi(l)$ with fixed $\pi$ and $l$ means that $q$ belongs to a certain residue class, say, $a_{l,\pi}$ modulo $\pi$.  Since $q>x^{1/4}$ and $\pi<x^{1/5}$, we have by the Brun-Titchmarsh inequality
	$$
	S_2=\sum_{y<\pi\leq x^{1/5}}\frac{1}{\pi}\sum_{l<x^{3/4}} \frac{x/l}{\varphi(\pi)\log (x/(\pi l))}\ll \frac{x}{\log x}\sum_{y<\pi\leq x^{1/5}}\frac{1}{\pi^2}\sum_{l<x^{3/4}}\frac1l \ll x,
	$$
	as desired (actually the right--hand side above is even $O(x/(y\log y))=o(x)$ as $x\to\infty$). Now the claim (\ref{1}) follows from the obtained bounds for $S_1,S_2,S_3$. This concludes the proof.
\end{proof}

\medskip 

Now we are ready to prove the theorem. Let $C>0$ be a large enough constant to be chosen later. Let $B$ be the set of $m\leq x^{7/15}$ such that $m=kqr$, where
\begin{itemize}
	\item[(i)] $k\leq x^{1/60}$ and $\varphi(k)$ is divisible by $d^2$ for any $d\leq y$;
	
	\item[(ii)] $r\in(x^{1/15}, x^{1/12}]$ is a prime;
	
	\item[(iii)] $q\in(x^{7/20}, x^{11/30}]$ is a prime;
	
	\item[(iv)] $m$ obeys the properties from Lemmas \ref{lem5.1}, \ref{lem5.2}, \ref{lem5.3}, and Lemma \ref{lem5.4} with the constant $C$; 
	
\end{itemize}

\smallskip 

Now we define the set $A$ of $n\in (x/2,x]$ of the form $n=mp$, where $m\in B$ and $p\in (x/(2m), x/m]$ is a prime. Our first aim is to show that $|A|\gg x$. Note first that working only with $m$ given by (i), (ii), (iii), we have that for such $m$, $q=P^{+}(m),~r=P^+(m/P^+(m))$. Thus, given $m\in B$, we have that $r,~q$ and hence also $k$ are uniquely determined. 
Then 
\begin{equation}\label{eq:recip}
\sum_{m\in B} \frac{1}{m}=\left(\sum_{(i)} \frac{1}{k} \right)\left(\sum_{(ii)} \frac{1}{r}\right)\left(\sum_{(iii)} \frac{1}{q}\right).
\end{equation}
In the above, (i), (ii), (iii) in the ranges of mean the parameters obey the conditions indicated  at (i), (ii), (iii), respectively; we also assume that $krq$ obeys (iv). The sum (i) is $\gg \log x$, since the sum $1/k$ over $k$'s failing one of the properties from Lemma \ref{lem5.2} is $o(\log x)$. Further, the sums (ii) and (iii) are 
$\gg 1$. Thus, the sum of reciprocal $m$'s satisfying (i), (ii), (iii) is $\gg \log x$. The $m$'s failing one of the properties from Lemma \ref{lem5.1}, \ref{lem5.2}, \ref{lem5.3} have the property that the sum of their reciprocals up to $x^{7/15}$ is $o(\log x)$. So, such $m$'s can be ignored in \eqref{eq:recip} and we get that the sum of the reciprocal $m$'s satisfying (i), (ii), (iii), and (iv) is a positive proportion of $\log x$. Let $c_1$ be such that the sum \eqref{eq:recip} exceeds $c_1\log x$. Next, note that $P^{+}(m)>m^{1/2}$. Indeed, this is equivalent to $q>rk$, which holds by examining the intervals for $q,~r,~k$. By Abel's summation formula in Lemma \ref{lem5.4}, for any fixed $C$, we have that 
\begin{equation}\label{2} 
	\sum_{\substack{m\leq x \\ P^+(m)>m^{1/2}\\ \sum_{\pi| m+\varphi(m)}\frac{1}{\pi}>C }}\frac{1}{m} \ll \frac{\log x}{C}.
\end{equation}
Let $c_2$ be the constant implied by the last $\ll $ above. Thus, taking $C>2c_2/c_1$, we see that the sum of reciprocals of $m$'s satisfying (i)--(iv) and also (v) for this value of $C$ is at least $0.5c_1\log x$. Thus,
$$
\sum_{m\in B}\frac1m \gg \log x.
$$ 
Note also that since $m\leq x^{1/2-1/30}$, we have $\log (x/m) \asymp \log x$. Therefore by the Prime Number Theorem and the previous bound
\begin{equation}\label{4} 
	|A|\geq \sum_{m\in B}\left(\pi(x/m)-\pi(x/(2m))\right) \asymp \frac{x}{\log x} \sum_{m\in B}\frac1m  \asymp x,
\end{equation}
as needed.

Let
$$
E:=\#\{(n,n')\in A^2: n+\varphi(n)=n'+\varphi(n') \}.
$$
To prove the theorem, it is enough to show that
\begin{equation}\label{5} 
	E \ll x.	
\end{equation}
Indeed, then by the Cauchy-Schwarz inequality and (\ref{4})
$$
\#\{u\leq 2x: u=n+\varphi(n) \mbox{ for some } n\in A\} \geq \frac{|A|^2}{E} \gg x,
$$
and the claims follows. Further, clearly,
\begin{equation}\label{6}
	E=|A|+R,	
\end{equation}
where $R$ corresponds to the solutions with $n\neq n'$. Since $\varphi(n)=\varphi(mp)=\varphi(m)(p-1)$, we can rewrite the equation $n+\varphi(n)=n'+\varphi(n')$ with distinct $n=mp$ and $n'=m'p'$ from $A$ as
\begin{equation}\label{7} 
	(m+\varphi(m))p-\varphi(m)=(m'+\varphi(m'))p'-\varphi(m').	
\end{equation}
For a number $d\leq \log x$ with $P^+(d)\leq y$, we let 
$$
B_d:=\{m\in B: D(m,y)=d\}.
$$
The definition of the set $B$ implies that the largest $y$-smooth divisor of the left-hand side of (\ref{7}) is equal to $D(m+\varphi(m),y)=D(m,y)$, and the largest $y$-smooth divisor of its right-hand side is $D(m'+\varphi(m'),y)=D(m',y)$. Hence, the equation (\ref{7}) can hold if $m$ and $m'$ belong to $B_d$ for some $d$. Therefore, if we denote by $I_{m,m'}$ the number of solutions of (\ref{7}) in primes $p\leq x/m$ and $p'\leq x/m'$, we can write
$$
R\leq \sum_{\substack{P^+(d)\leq y\\d\leq \log x}}E_d,
$$
where
$$
E_d:=\sum_{m,m'\in B_d}I_{m,m'}.
$$
Now we see that it is enough to establish the bound
\begin{equation}\label{8} 
	E_d \ll \frac{x}{d\log y}
\end{equation}
uniformly in $d$. Once this it is done, we have
$$
E=|A|+\sum_{\substack{P^+(d)\leq y\\ d\leq \log x}}E_d \ll x+\sum_{P^+(d)\leq y }\frac{x}{d\log y} = x+\frac{x}{\log y}\prod_{p\leq y}\left(1-1/p\right)^{-1} \ll x,
$$
and (\ref{5}) and the claim follows.

\medskip

We fix a number $d\leq \log x$ with $P^+(d)\leq y$ for the rest of the proof. Our goal is to estimate $E_d$. We rewrite the equation (\ref{7}) as
\begin{equation}\label{9} 
	(m+\varphi(m))p-(m'+\varphi(m'))p'=\varphi(m)-\varphi(m').	
\end{equation}
We claim that may assume $\varphi(m)\neq \varphi(m')$ (in particular, $m\neq m'$). Indeed, since both $m$ and $m'$ are at most $x^{1/2-1/15}$, we would otherwise get
$$
(m+\varphi(m))p=(m'+\varphi(m'))p'
$$
and $p=p'$ as the largest prime divisor of this number, and so $m=m'$ and $n=n'$ which contradicts our assumption. So,
$$
E_d= \sum_{\substack{m, m'\in B_d \\ \varphi(m)\neq \varphi(m')}} I_{m,m'}.
$$ 
Let us fix any distinct $m$ and $m'$ from $B_d$. As we noted before,
$$
d=D(m,y)=D(m',y)=D(m+\varphi(m),y)=D(m'+\varphi(m'),y).
$$
Let 
$$
h:=h(m,m')=\frac{\gcd(m+\varphi(m),m'+\varphi(m'))}{d}.
$$ 
We see that all prime divisors of $h$ are greater than $y$. Now note that the equation (\ref{9}) has no solutions if its right-hand side is not divisible by $h$, and otherwise is equivalent to   
\begin{equation}\label{10} 
\frac{m+\varphi(m)}{dh}p-\frac{m'+\varphi(m')}{dh}p'=\frac{\varphi(m)-\varphi(m')}{dh}. 
\end{equation}
Let us denote
$$
M':=M'(m,m')=(m+\varphi(m))(m'+\varphi(m'))|\varphi(m)-\varphi(m')|.
$$
Selberg's sieve implies  
\begin{multline*}
	I_{m,m'} \ll \frac{x/(dh)}{((m+\varphi(m))/(dh))((m'+\varphi(m'))/(dh))(\log x)^2}\prod_{p|M'}(1+1/p) \\
	\ll  \frac{xdh}{mm'(\log x)^2}\prod_{p|\varphi(m)-\varphi(m')}(1+1/p),
\end{multline*}
since
\begin{equation}\label{10.5} 
\prod_{p|m+\varphi(m)}(1+1/p) \leq \exp\left(\sum_{p|m+\varphi(m)}\frac1p\right) \leq e^C \ll 1
\end{equation}
by the definition of the set $B$, and the same bound holds for $m'+\varphi(m')$. Therefore,
$$
E_d \ll \frac{xd}{(\log x)^2}\sum_{m,m'\in B_d} \frac{h}{mm'}\prod_{p|\varphi(m)-\varphi(m')}(1+1/p). 
$$
Clearly,
$$
\prod_{\substack{p|\varphi(m)-\varphi(m')\\ p\leq \log_2 x}}(1+1/p) \leq \prod_{p\leq \log_2 x }(1+1/p) \ll \log y
$$
and, since the number $|\varphi(m)-\varphi(m')|\leq x^{7/15}$ has $O(\log x/\log\log x)$ prime factors,
$$
\prod_{\substack{p|\varphi(m)-\varphi(m')\\ p>\log x}}(1+1/p) \leq \exp\left(\sum_{\substack{p|\varphi(m)-\varphi(m')\\ p>\log x}}\frac1p\right)  \ll 1.
$$
So, if we denote
\begin{equation}\label{11} 
	M:=M(m,m')=\prod_{\substack{p|\varphi(m)-\varphi(m')\\ \log_2 x<p\leq \log x}}(1+1/p),
\end{equation}
it is enough to show that
\begin{equation}\label{12}
	\frac{xd\log y}{(\log x)^2}\sum_{\substack{m,m'\in B_d\\m\neq m'}}\frac{hM}{mm'} \ll \frac{x}{d\log y}.
\end{equation}

Now we consider three different cases depending on how large $h=h(m,m')$ is.

\subsection*{1. The case $h>x^{1/3}$} 

Recall that $m\in B$ has the form $m=krq$ with prime $r$ and $q$, and let us denote $l=kr$. Using the bounds for $k$ and $r$, we see that $l\leq x^{1/10}$. We have $m=lq$ and 
\begin{equation}\label{13}
m+\varphi(m)=(l+\varphi(l))q-\varphi(l) \equiv 0 \pmod{h}.	
\end{equation}

We claim that $\gcd(l+\varphi(l),h)=1$. Indeed, if for some $p|h$ (and thus $p>y$) we would have $p|l+\varphi(l)$, thus $p|\varphi(l)$ by the above congruence, and hence $p|l$; but then $p$ divides both $m$ and $\varphi(m)$, which is impossible by the definition of the set $B$ (property (ii) of Lemma \ref{lem5.2}).

Thus, the above congruence means that, for a fixed $l$ and $h$, the number $q$ is contained in a certain residue class modulo $h$, say
$$
q\equiv a_{h,l} \pmod{h}.
$$
Similarly, $\gcd(l'+\varphi(l'),h)=1$ and,  for a fixed $l'$, we have $q'\equiv a_{h,l'} \pmod{h}$. 

Since $h$ divides all three numbers $\varphi(m)-\varphi(m')$, $m+\varphi(m)$, and $m'+\varphi(m')$, we have $m\equiv m'\pmod{h}$. Then (\ref{13}) implies
$$
\frac{l\varphi(l)}{l+\varphi(l)} \equiv lq \equiv l'q'\equiv \frac{l'\varphi(l')}{l'+\varphi(l')} \pmod{h},
$$
or
$$
(l'+\varphi(l'))l\varphi(l) - (l+\varphi(l))l'\varphi(l') \equiv 0 \pmod{h}.
$$
But since $\max\{l,l'\}\leq x^{1/10}$, the absolute of the left-hand side is at most $2x^{3/10}<x^{1/3}<h$. It means that
$$
\frac{l\varphi(l)}{l+\varphi(l)}=\frac{l'\varphi(l')}{l'+\varphi(l')}.
$$
Now recall that $l=kr$ and $l'=k'r'$, where $r=P^+(l)$ and $r'=P^+(l')$. Clearly, $r$ does not divide $\varphi(l)=\varphi(k)(r-1)$ and $l+\varphi(l)=(k+\varphi(k))r-\varphi(k)$, since $k$ is small. Thus, if we write
$$
\frac{l\varphi(l)}{l+\varphi(l)}=\frac{a}{b}=\frac{l'\varphi(l')}{l'+\varphi(l')}
$$
with $\gcd(a,b)=1$, then $r=r'=P^+(a)$. Besides, since $d$ is $y$-smooth, $d^2$ divides $\varphi(k)$ by the properties (i) and (iii) of Lemma \ref{lem5.2}, and hence, $d^2$ divides $\varphi(l)$ as well. Then it is not hard to see that $\gcd(l\varphi(l),l+\varphi(l))=d$, so actually 
$$
b=(l+\varphi(l))/d=(l'+\varphi(l'))/d.
$$  
So, we have $r=r'$ and 
$$ 
l+\varphi(l)=(k+\varphi(k))r-\varphi(k)=(k'+\varphi(k'))r-\varphi(k')=
l'+\varphi(l').
$$
Thus,
$$
(k+\varphi(k)-k'-\varphi(k'))r=\varphi(k)-\varphi(k'),
$$
so $r$ divides $\varphi(k)-\varphi(k')$ which implies first that $\varphi(k)=\varphi(k')$ (because $|\varphi(k)-\varphi(k')|$ is much smaller than $r$) and next that  $k+\varphi(k)=k'+\varphi(k')$; that is, $k=k'$. 

So, we get $l=l'$ if $h(m,m')>x^{1/3}$. This means that $q\equiv q' \pmod{h}$. Since $l=l'$, we have $q\neq q'$; without loss of generality we assume that $q>q'$. Then using the trivial bound for $M$ and recalling that $d\leq \log x$, we get
$$
W_1:=\frac{xd\log y}{(\log x)^2}\sum_{\substack{m,m'\in B_d:\\m\neq m',\, h(m,m')>x^{1/3}}}\frac{hM}{mm'} \ll x^{1+o(1)}\sum_{l, q', h, q}\frac{h}{l^2q'q}.
$$ 
Now for fixed $l$, $q'$, and $h \,| \, (l+\varphi(l))q-\varphi(l) = m'+\varphi(m')$,
the sum of $1/q$ over $q\equiv a_{h,l} \pmod{h}$ is at most $O((\log x)/h)$, since $q>q'$ and $q'\equiv a_{h,l} \pmod{h}$ imply that we in fact sum values of $q$ that are greater than $h$ in an arithmetic progression modulo $h$. Since for fixed $l$ and $q'$ there are at most $\tau( (l+\varphi(l))q-\varphi(l)) \leq x^{o(1)}$ possible values of $h$, we get
$$
W_1 \ll x^{1+o(1)}\sum_{l,q'}\frac{1}{q'l^2}.
$$
Now the sum of $1/q'$ is $O(1)$ and $\sum_{l}1/l^2=\sum_{k}1/k^2\sum_r1/r^2 \ll x^{-1/15}$ because of the range of $r$. Since $d \leq \log x$, we obtain
$$
W_1 \ll  x^{14/15+o(1)} \ll \frac{x}{d\log y},
$$
as desired.

\subsection*{2. The case $x^{1/20}<h\leq x^{1/3}$} 

Recall that we have $m=krq$ and
$$
m\equiv -\varphi(m) \equiv -\varphi(m') \equiv m' \pmod{h}.
$$
Note also that $\gcd(m,\varphi(m))=d$ for $m\in B_d$ by the properties from Lemma \ref{lem5.2}. Since all prime factors of $h$ are greater than $y$ and $h$ divides $m+\varphi(m)$, we thus have $\gcd(k,h)=\gcd(\varphi(k),h)=1$. So, the pair of congruences 
$$
qrk \equiv -(q-1)(r-1)\varphi(k)\equiv m' \pmod{h} 
$$
determine $u:=qr$ and $v:=q+r$ when $k$ and $m'$ are given. Since all prime factors of $h$ are odd and the congruence $x^2\equiv a \pmod{p^{\alpha}}$ has at most two solutions for any odd prime power $p^{\a}$, there are at most $2^{\o(h)}\leq \tau(h)$ solutions of the congruence $t^2-vt+u\equiv 0 \pmod{h}$. So there are at most $\tau(h)$ pairs $(a,b) \pmod h$ for which $q\equiv a\pmod h$ and $r\equiv b\pmod h$ are possible. Let $S(m',k,h)$ denote the set of all such possible pairs $(a,b)$. Now we note that for fixed $m', k, h$ and $(a,b)\in S(m',k,h)$ we get, by the Brun-Titchmarch inequality and the fact that $q>x^{7/20}\geq h^{1+\delta}$ for some small $\delta$ (say, $\delta=1/60$ works), 
$$
\sum_{q\equiv a \pmod h}\frac1q \ll \frac{1}{\varphi(h)}\ll \frac1h;
$$
the last inequality here is due to the fact that $h$ divides $m+\varphi(m)$ and (\ref{10.5}). Also, 
$$
\sum_{r \equiv b\pmod h} \frac1r \ll \frac{1}{x^{1/15}}+\frac{\log x}{h} \ll \frac{\log x}{x^{1/20}} \, . 
$$
In the above, $1/x^{1/15}$ accounts for the smallest term of the progression $b\pmod h$ which could be smaller than $h$ but must be a prime $r>x^{1/15}$ and $(\log x)/h$ accounts for the remaining sum even ignoring the condition that such $r$'s are prime. 
Therefore, again using the trivial upper bounds for $d$ and $M$, we obtain
\begin{multline*}
	W_2:=\frac{x\log y}{(\log x)^2}\sum_{\substack{m,m'\in B_d\\m\neq m',\, x^{1/20}<h(m,m')\leq x^{1/3}}}\frac{dhM}{mm'} \ll \\ x^{1+o(1)}\sum_{m',h,k}\frac{h}{m'k}\sum_{(a,b)\in S(m',k,h)}\sum_{r \equiv b\pmod h} \frac1r\sum_{q\equiv a \pmod h}\frac1q \ll \\
	x^{19/20+o(1)}\sum_{m',h,k}\frac{\tau(h)}{m'k}\ll x^{19/20+o(1)}\sum_{m'\in B_d}\frac{1}{m'}\sum_{h| m'+\varphi(m')}\tau(h) \ll \\
	x^{19/20+o(1)}\sum_{m'\in B}\frac{\tau(m'+\varphi(m'))^2}{m'} \ll x^{19/20+o(1)}\ll \frac{x}{d\log y},
\end{multline*}
as desired.

\bigskip 

\subsection*{3. The case $h\leq x^{1/20}$.} 

We will argue as in the beginning of previous case, but now pay attention to $M(m_1,m_2)$. By the properties stated in Lemmas \ref{lem5.3} and \ref{lem5.4} and the inequality $1+t\leq e^t$, we see that
$$
\prod_{\substack{\pi|\varphi(m)(m'+\varphi(m')) \\ \pi>\log_2 x}}(1+1/\pi) \ll 1.
$$ 
Thus, 
$$
M(m,m') \ll \prod_{\substack{\pi|\varphi(m)-\varphi(m')\\ \log_2 x<\pi\leq \log x \\ \pi \,\nmid\, \varphi(m)(m'+\varphi(m')}}(1+1/\pi).
$$
Now we fix $m',k,h$ and a pair $(a,b)\in S(m',k,h)$. Let us define
$$
f(qr)=f_{m',h,k,a,b}(qr)=\sum_{\substack{\pi |\varphi(m)-\varphi(m')\\ \log_2 x<\pi \leq \log x \\ \pi \,\nmid\, \varphi(m)(m'+\varphi(m')}} \frac{1}{\pi}.
$$
We will show that the main contribution to (\ref{12}) comes from $qr$ for which $f(qr)\leq 1$; note that in this case $M(m,m') \ll 1$.

Note that if for a fixed $r$ we have $\pi|\varphi(kqr)-\varphi(m')$, then 
$$
q\varphi(kr) \equiv \varphi(kr)+\varphi(m') \pmod \pi; 
$$
since $\pi$ does not divide $\varphi(kr)$, we see that $q$ belongs to some residue class modulo $h$, say $c_{\pi,kr,m'} \pmod \pi$. Note also that $h | m'+\varphi(m')$ and $\pi$ does not divide $m'+\varphi(m')$, and thus $\pi$ and $h$ are coprime. Since we also assume $q\equiv a \pmod h$  and $r\equiv b \pmod h$, under the above assumptions $q$ belongs to a certain residue class modulo $\pi h\leq x^{1/20}\log x$. Therefore, by the Brun-Titchmarsh inequality and the fact that $\varphi(h)\gg h$ (which follows from $\varphi(m'+\varphi(m'))\gg m'+\varphi(m')$, since $h$ divides $m'+\varphi(m')$),
\begin{equation}\label{14}
	\sum_{q,r}\frac{f(qr)}{qr} \ll \sum_{\pi}\frac{1}{\pi}\sum_{r}\frac1r\sum_q\frac1q \ll \sum_{\pi}\frac{1}{\pi\varphi(\pi h)}\sum_{r}\frac1r \ll \frac{1}{h^2}\sum_{\pi>\log_2 x}\frac{1}{\pi^2} \ll \frac{1}{h^2\log_2 x\log y}.	
\end{equation}
Now we bound $M(m,m')$ by $O(1)$ if $f(qr)\leq 1$ and by $(\log_2 x)/\log y$ otherwise; this is a trivial bound as 
$$
M\le \exp\left(\sum_{\log_2 x < \pi<\log x} \frac{1}{\pi}\right)=\exp(\log_3 x-\log_4 x+o(1))\ll \exp\left(\log\left(\frac{\log_2 x}{\log y}\right)\right)\ll \frac{\log_2 x}{\log y}.
$$
We get
\begin{multline*}
W_3:=\frac{xd\log y}{(\log x)^2}\sum_{\substack{m,m'\in B_d:\\m\neq m', \, h\leq x^{1/20}}}\frac{hM}{mm'} \ll \\
\frac{xd\log y}{(\log x)^2} \sum_{m',k,h}\frac{h}{m'k}\sum_{(a,b)\in S(m',k,h)} \left(\sum_{\substack{q,r \\ f(qr)\leq 1}}\frac{1}{qr}+ \sum_{\substack{q,r \\ f(qr)>1}}\frac{\log_2 x}{qr\log y}\right), 
\end{multline*}
where the inner sums are over $r\equiv b \pmod h$ and $q \equiv a \pmod h$. 
From (\ref{14}) we obtain
$$
\sum_{\substack{q,r \\ f(qr)>1}}\frac{\log_2 x}{qr\log y} \leq \frac{\log_2 x}{\log y}\sum_{q,r}\frac{f(qr)}{qr} \ll \frac{1}{h^2(\log y)^2},
$$
and thus
$$
\sum_{\substack{q,r \\ f(qr)\leq 1}}\frac{1}{qr} \ll \frac{1}{h^2}.
$$
Since $|S(m',k,h)|\leq \tau(h)$, we have
$$
W_3 \ll \frac{xd\log y}{(\log x)^2} \sum_{m',k}\frac{1}{m'k} \sum_{h| m'+\varphi(m')}\frac{\tau(h)}{h}.
$$
Now
\begin{multline*}
\sum_{h|m'+\varphi(m')}\frac{\tau(h)}{h} =\sum_{h|m'+\varphi(m')}\frac1h\sum_{d|h}1=\sum_{d|m'+\varphi(m')}\sum_{\substack{h|m'+\varphi(m')\\m\equiv 0\pmod d}}\frac1h\\
=\sum_{d|m'+\varphi(m')}\frac1d\sum_{u|(m'+\varphi(m'))/d}\frac1u\leq \left(\sum_{h|m'+\varphi(m')}\frac1h\right)^2 \ll e^{2C} \ll 1
\end{multline*}
due to (\ref{10.5}). Hence,
$$
W_3 \ll \frac{xd\log y}{(\log x)^2}\sum_{m'\in B_d}\frac{1}{m'}\sum_k\frac1k,
$$
where the latter summation is over $k\leq x^{1/60}$ with $D(k,y)=d$. We have 
$$
\sum_{m'\in B_d}\frac{1}{m'}\leq \frac1d\sum_{\substack{t\leq x^{7/15}/d\\ P^-(t)>y}}\frac1t \ll \frac1d\prod_{y<p\leq x^{7/15}}\left(1+\frac1p+\frac{1}{p^2}+\cdots\right) \ll \frac{\log x}{d\log y}.
$$ 
Similarly, the sum $1/k$ over $k$ which occur from $m\in B_d$ is $O((\log x)/(d\log y))$. Thus
$$
W_3 \ll \frac{x}{d\log y},
$$ 
as desired. Now (\ref{12}) follows and the lower bound in the theorem is proved.

\subsection{Proof of the upper bound}

Now we prove the upper bound. We will need the fact the estimate
\begin{equation}\label{5.5}
	\Phi_x(\l)=\Phi(\l)+O\left(\frac{1}{\log x}\left(\frac{\log\log x}{\log\log\log x}\right)^2\right)
\end{equation}
holds uniformly on $\l\in[0,1]$ (see \cite[\S4.8]{P}, and also \cite{F1}, \cite{F2}). 

The main observation here is that if $k>\l x$ and $\varphi(k)/k>(1-\l)/\l$, then $k+\varphi(k)>x$. Let us take numbers $\{\l_i\}_{i=1}^r$ such that 
$$
1/2<\l_1<\cdots <\l_r<1
$$ 
and $\l_i-\l_{i-1}$ is of order $1/\log\log x$ (say) for each $i$; thus, $r\asymp \log\log x$. Then, writing $\nu_i=(1-\l_i)/\l_i$ and using (\ref{5.5}), we get
\begin{multline*}
	x-N^+_{\varphi}(x) \geq \sum_{i=2}^r\#\{\l_{i-1}x< k\leq \l_ix: \varphi(k)/k>\nu_{i-1} \} \\
	= \sum_{i=2}^r\bigg(\l_ix(1-\Phi_{\l_ix}(\nu_{i-1})) - \l_{i-1}x(1-\Phi_{\l_{i-1}x}(\nu_{i-1}))\bigg)\\
	=\sum_{i=2}^r\bigg(\l_ix(1-\Phi(\nu_{i-1})) - \l_{i-1}x(1-\Phi(\nu_{i-1}))\bigg)+O\left(\frac{rx}{(\log x)^{1-o(1)}}\right)\\
	=x\sum_{i=2}^r(\l_i-\l_{i-1})(1-\Phi(\nu_{i-1}))+O\left(\frac{x}{(\log x)^{1-o(1)}}\right)\\
	=x\int_{1/2}^1\left(1-\Phi\left(\frac{1-\l}{\l}\right)\right)d\l+o(x)=x\left(1/2-\int_0^1\frac{\Phi(t)dt}{(1+t)^2}\right)+o(1)).
\end{multline*}
So,
\begin{equation*}
N^+_{\varphi}(x)\leq x\left(1/2+\int_0^1\frac{\Phi(t)dt}{(1+t)^2}+o(1)\right).
\end{equation*}

This concludes the proof.

\end{document}